\documentclass[envcountsame]{llncs}
\usepackage{llncsdoc}
\usepackage{graphicx}
\usepackage[all]{xy}
\usepackage{amssymb}
\usepackage{amsmath}
\usepackage{hyperref}

\begin{document}
\title{Non-Commutative Homometry in the Dihedral Groups}
\author{Gr\'egoire Genuys}
\institute{ UPMC/IRCAM/CNRS \\
\email{ gregoire.genuys@ircam.fr}
}
\maketitle

\begin{abstract}
The paper deals with the question of homometry in the dihedral groups $D_{n}$ of order $2n$. These groups have the specificity to be non-commutative. It leads to a new approach as compared as the one used in the traditional framework of the commutative group $ \mathbb{Z}_{n}$. We give here a musical interpretation of homometry in $D_{12}$ using the well-known neo-Riemannian groups, some computational results concerning enumeration of homometric sets for small values of $n$, and some properties disclosing important links between homometry in $\mathbb{Z}_{n}$ and homometry in $D_{n}$. Finally we propose an extension of musical applications for this non-commutative homometry.
\end{abstract}

\section{Introduction}

This paper gathers results drawn from two of our works: the article "Homometry in the Dihedral Groups: Lifting Sets from $ \mathbb{Z}_{n}$ to $ D_{n}$" (\cite{Genuys1}) and the Ph.D. thesis "Etude de deux concepts math\'ematico-musicaux : l'homom\'etrie non-commutative et les distances d'accords." (\cite{Genuys2}). Our idea is to propose an extended version of \cite{Genuys1} adding some results of \cite{Genuys2}.

The concept of homometry first appeared in the 1930s in the field of cristallography. The question was to determine the structure of a crystal from its X-ray diffraction pattern. This type of measurement is directly related to the intensity of the Fourier transform of the crystallographic structure, but the phase information is lost in the process. The problem was therefore to know if a complete reconstruction of the structure of the crystal was possible. This problem later found applications in various fields, such as music theory where the question was to characterize a set of notes (a chord, a melody) from the intervals that compose it. Homometry has been studied through different approaches: group theory \cite{Lachaussee}, Fourier transform \cite{Amiot}, distribution theory \cite{Mandereau1,Mandereau2}, etc. and is an open field of research.

The classical way to model the $n$-tone equal temperament in musical set theory is to consider the cyclic group $\mathbb{Z}_{n}=\mathbb{Z}/n\mathbb{Z}$ as the set of pitch classes, for instance $\mathbb{Z}_{12}=\{ 0=C, 1=C^{\sharp}, ..., 11=B\}$. Following David Lewin's constructions described and systematized in \cite{Lewin}, for any subsets $A$ and $B$ in $\mathbb{Z}_{n}$, one can consider the interval function \textbf{ifunc}($A,B$) whose components are 
$$
\textbf{ifunc}(A,B)(k)=\sharp \{ (a,b) \in A \times B \mid b-a = k\}
$$
for $k\in \mathbb{Z}_{n}$, and the interval vector \textbf{iv}($A$) whose components are defined by $\textbf{iv}(A)(k)=\textbf{ifunc}(A,A)(k)$. Two sets $A$ and $B$ in $\mathbb{Z}_{n}$ are \textit{homometric} if they have the same interval vector ($\textbf{iv}(A)=\textbf{iv}(B)$), meaning that they contain the same set of intervals. This is traditionally called \textit{Z-relation} and was mainly presented by Forte \cite{Forte}. In this paper we will only use the word homometry which refers to the same concept. The actions of transposition and inversion clearly do not change the interval vector of a set, hence two homometric sets which do not belong to the same set class modulo transpositions and inversions will be called \textit{non-trivial homometric sets}. A well-known example of a non-trivial homometric pair in $\mathbb{Z}_{12}$ is $(\{C,D^{\flat},E,G^{\flat}\}, \{C,D^{\flat},E^{\flat},G\})$, for which the interval vector is $[1,1,1,1,1,2,1,1,1,1,1]$. More detailled explanations can be found in \cite{Mandereau1} and \cite{Jed}.

The concept of group action of $(\mathbb{Z}_{n},+)$ on itself by translation (where $n\in \mathbb{Z}_{n} $ acts on $a\in \mathbb{Z}_{n}$ by $n+a$) can also be used to define the interval vector. Thus we call \textit{interval} between $a$ and $b$, written \textbf{int}$(a,b)$, the element $n$ such that $n+a=b$, and \textbf{iv}($A$)$(k)=\sharp \{ (a,b) \in A^2 \mid \textbf{int}(a,b) = k\}$ for $k\in \mathbb{Z}_{n}$. In a more general setting and following Lewin's idea of generalized intervals \cite{Lewin}, one can consider homometry in the context of any simply transitive group action. For instance it is well-known that the $T/I$-group and the neo-Riemannian $PLR$-group both act simply transitively on the set $S$ of major and minor triads. We recall that the $T/I$-group is generated by the transpositions $T_p(x)=p+x$ and the inversions $I_p(x)=T_pI_0(x)$ i.e. $I_p(x)=-x+p$, for $p\in \mathbb{Z}_{n}$. The $PLR$-group is generated by $P$, $L$, and $R$ which correspond respectively to the operations parallel, leading tone exchange and relative (see \cite{Crans} for more details). The interval between two triads $s_1$ and $s_2$ is the unique element of the group sending $s_1$ to $s_2$ for the chosen group action. If we use upper-case letters for major triads and lower-case letters for minor triads ($C$ is $C$-major and $c$ is $C$-minor) we obtain for instance in the context of the $T/I$-group: \textbf{int}$_{T/I}(c,E^{\flat})=I_{10}$ and in the context of $PLR$-group: \textbf{int}$_{PLR}(c,E^{\flat})=R$. For a given Generalized Interval System $(S,\text{IVLS},\textbf{int})$, we can thus define the interval vector of a subset $A$ of $S$ as
$$\textbf{iv}(A)(k)=\sharp \{ (a,b) \in A^2 \mid \textbf{int}(a,b) = k\}$$
for $k$ in $\text{IVLS}$, and two subsets of $S$ will be called homometric if they have the same interval vector. As an example of homometry for both the actions of the $T/I$-group and the $PLR$-group, consider the pair of sets $\{ c, D^{\flat}, E^{\flat}, e, a^{\flat} \}$ and $\{ c, E^{\flat}, e, F, a^{\flat} \}$. Figure \ref{figleftright1} shows some of the intervals between the elements of these sets, for the action of the $T/I$-group and the $PLR$-group respectively. It can clearly be seen that the same intervals $\{T_2, T_4, T_4, T_8, I_0, I_2, I_4, I_6, I_8, I_{10}\}$ are present in both sets, hence they have the same interval vector (all other intervals can be deduced by composition and/or inversion). Similarly, the same intervals $\{R, PL, PL, PL, LRL, LRP, RLP, LPR, PRP, PRLR\}$ are present in both sets leading to an identical conclusion for the interval vector.

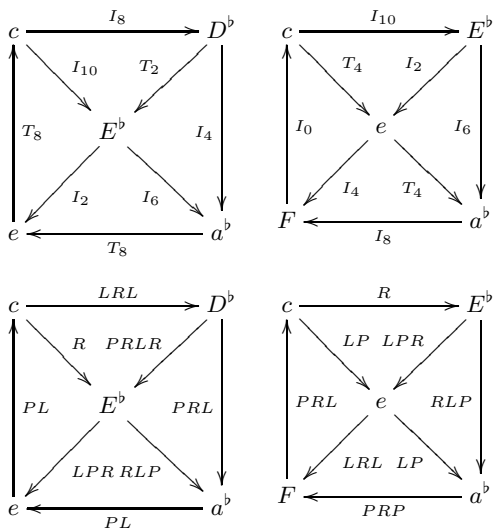
\begin{figure}[t!]
\[
\xymatrix{
c \ar[dr]^{I_{10}} \ar[rr]^{I_8} & & D^{\flat} \ar[dl]_{T_2} \ar[dd]_{I_4} \\
& E^{\flat} \ar[dl]^{I_2} \ar[dr]_{I_6} & \\
e \ar[uu]_{T_8} & & a^{\flat} \ar[ll]^{T_8} }
\quad
\xymatrix{
c \ar[dr]^{T_4} \ar[rr]^{I_{10}} & & E^{\flat} \ar[dl]_{I_2} \ar[dd]_{I_6} \\
& e \ar[dl]^{I_4} \ar[dr]_{T_4} & \\
F \ar[uu]_{I_0} & & a^{\flat} \ar[ll]^{I_8} }
\]

\[
\xymatrix{
c \ar[dr]^{R} \ar[rr]^{LRL} & & D^{\flat} \ar[dl]_{PRLR} \ar[dd]_{PRL} \\
& E^{\flat} \ar[dl]^{LPR} \ar[dr]_{RLP} & \\
e \ar[uu]_{PL} & & a^{\flat} \ar[ll]^{PL} }
\quad
\xymatrix{
c \ar[dr]^{LP} \ar[rr]^{R} & & E^{\flat} \ar[dl]_{LPR} \ar[dd]_{RLP} \\
& e \ar[dl]^{LRL} \ar[dr]_{LP} & \\
F \ar[uu]_{PRL} & & a^{\flat} \ar[ll]^{PRP} }
\]
\caption{Intervals in the $T/I$-group (top) and the $PLR$-group (bottom) for the two homometric sets $\{ c, D^{\flat}, E^{\flat}, e, a^{\flat} \}$ and $\{ c, E^{\flat}, e, F, a^{\flat} \}$.}
\label{figleftright1}
\end{figure}

It has been showed in \cite{Popoff} that the actions of the $T/I$-group and of the $PLR$-group on the set of major and minor triads can be understood respectively as the left and the right actions of the dihedral group $D_{12}$ on this set. Moreover it is well-known that $D_{12}$ is the semi-direct product $(\mathbb{Z}_{12},+) \rtimes (\mathbb{Z}_{2},\times)$. This lead us to the general topic of this paper, namely the study of homometry in the non-commutative dihedral groups of order $2n$, or equivalently the semi-direct products $(\mathbb{Z}_{n},+) \rtimes (\mathbb{Z}_{2},\times)$.

This paper is divided into four parts. We first recall the definition of $D_{n}$ as a semi-direct product and we define homometry for the right and left actions of this group. The link with the well-known musical case $n=12$ will be explained in the second part. In the third part we give the equations that characterize homometry and some results concerning enumeration. Finally we define in the last part the concept of \textit{lift} which bridges homometry in $\mathbb{Z}_{n}$ and homometry in $D_{n}$, and we give our main results using the discrete Fourier transform.

\section{The Dihedral Group $D_{n}$ as a Semi-direct Product}

The dihedral group $D_{n}$ is the group of symmetries of a regular $n$-gon, using rotations and reflections. It can be expressed as the semi-direct product $(\mathbb{Z}_{n},+) \rtimes (\mathbb{Z}_{2},\times)$, where $ \mathbb{Z}_{2}= \{ \pm 1 \}$. Its elements are the pairs $(k,\epsilon)$ where $k\in \mathbb{Z}_{n}$ and $\epsilon = \pm 1$, with the identity element being $(0,1)$, and multiplication between two elements being given by the equation
\begin{equation}
\label{equ_mult}
(k,\epsilon)(l,\eta)=(k+\epsilon l, \epsilon \eta).
\end{equation}
The inverse of an element $(k, \epsilon)$ is $(k, \epsilon)^{-1}=(-k\epsilon, \epsilon)$.
As a non-commutative group $D_{n}$ acts on itself by right or left multiplication: thus $(l, \eta)$ acting on $(k,\epsilon)$ on the right leads us to Eq. (\ref{equ_mult}), whereas $(l, \eta)$ acting on $(k,\epsilon)$ on the left leads us to
\begin{equation}
(l, \eta)(k,\epsilon)=(l+\eta k, \eta \epsilon).
\end{equation}

This allows us to define the intervals between any two pairs $(k_1, \epsilon_1)$ and $(k_2, \epsilon_2)$ of $D_{n}$. The left interval is the unique element $(l,\eta)$ in $D_{n}$ such that $(l,\eta)(k_1, \epsilon_1)=(k_2, \epsilon_2)$, whereas the right interval is the unique element $(l,\eta)$ such that $(k_1, \epsilon_1)(l,\eta)=(k_2, \epsilon_2)$. Thus we obtain two functions $^l\textbf{int} \colon D_{n} \times D_{n} \to D_{n}$ and $^r\textbf{int} \colon D_{n} \times D_{n} \to D_{n}$, called \textit{interval functions} and defined as 
\begin{equation}
^l\textbf{int}((k_1, \epsilon_1),(k_2, \epsilon_2))=(k_2-\epsilon_2/\epsilon_1 k_1, \epsilon_2/\epsilon_1),
\end{equation}
\begin{equation}
^r\textbf{int}((k_1, \epsilon_1),(k_2, \epsilon_2))=((k_2-k_1)/\epsilon_1, \epsilon_2/\epsilon_1).
\end{equation}

The left interval vector $^l\textbf{iv}(A)$ and right interval vector $^r\textbf{iv}(A)$ of a set $A$ in $D_{n}$ are then defined as
\begin{align*}
^{l,r}\textbf{iv}(A)((l,\eta))=\sharp \{ ((k_1, \epsilon_1),(k_2, \epsilon_2)) \in A^2 \mid  {^{l,r}\textbf{int}}((k_1, \epsilon_1),(k_2, \epsilon_2)) = (l,\eta)\}
\end{align*}
for $(l,\eta)\in D_{n}$. We say that two sets in $D_{n}$ are \textit{homometric for the left (resp. for the right) action} (or simply \textit{left-/right-homometric}) if they have the same left (resp. right) interval vector. It is easy to see that any left (resp. right) action preserves the right (resp. left) intervals. It is natural to search other operations in $D_{n}$ that preserve right or left intervals. Recall that in $\mathbb{Z}_{n}$ we considered sets modulo translations and inversions. So we will first check if the inversion in $D_{n}$ is interval preserving or not. In fact it is not. For instance the set $\lbrace(0,1),(2,-1),(3,1)\rbrace$ in $D_{12}$ has the following interval vectors:
\begin{align*}
^r \textbf{iv}&: [(0, 1), (2, -1), (3, 1), (2, -1), (0, 1), (11, -1), (9, 1), (11, 11),(0, 1)],\\
^l\textbf{iv}&:[(0, 1), (2, -1), (3, 1), (2, -1), (0, 1), (5,-1), (9, 1), (5, -1),(0, 1)].
\end{align*}
The inverse $\lbrace(0,1),(2,-1),(-3,1)\rbrace$ has different interval vectors:
\begin{align*}
^r \textbf{iv}&: [(0, 1), (2,-1), (9, 1), (2,-1), (0, 1), (5, -1), (3, 1), (5, -1), (0, 1)],\\
^l \textbf{iv}&:[(0, 1), (2, -1), (9, 1), (2, -1), (0, 1), (11, -1), (3, 1), (11, -1), (0, 1)].
\end{align*}

Obviously it can still exist other interval preserving operations. That is why we will study the group of automorphisms of the dihedral group, and look for interval preserving operations inside this group.

\begin{proposition}
The group $\mathcal{A}ut(D_{n})$ of automorphisms of the dihedral group $D_{n}$ is isomorphic to $\mathbb{Z}_{n} \rtimes \mathbb{Z}_{n}^{\star}$, where $\mathbb{Z}_{n}^{\star}$ is the group of invertible elements of $\mathbb{Z}_{n}$. Then it has $\phi(n)n$ elements, where $\phi(n)$ is the number of $k \in \lbrace1,...,n-1\rbrace$ coprime with $n$.
\end{proposition}

This is a known result. If we use the following presentation for the dihedral group:
\begin{align*}
D_{n}&=\langle r,c \mid r^n=1, c^2=1, crc=r^{-1} \rangle \\
               &=\lbrace 1, r, r^2,..., r^{n-1}, c,cr,...,cr^{n-1}\rbrace,
\end{align*}
where the $r^p$ designate the rotations and the $cr^q$ designate the reflections, an automorphism $\pi=(l,k)\in \mathbb{Z}_{n} \times \mathbb{Z}_{n}^{\star}$ acts on $r^p$ by $\pi(r^p)=r^{kp}$ and on $cr^q$ by 
$$
\pi(cr^q)=\pi(c)\pi(r^q)=cr^lr^{kq}=cr^{kq+l}.
$$
In the point of view of semi-direct products, $r$ corresponds to $(1,1)$ and $c$ to $(0,-1)$, consequently $\pi=(l,k)$ acts on $(p,1)$ by $(kp,1)$ and on $(q,-1)$ by $(kq+l,-1)$. With this formulation we can now look at the interval preserving operations in $\mathcal{A}ut(D_{n})$. The result is that there is no such operation.

\begin{proposition}
\label{propnointervalpreserving}
There is no element of $\mathcal{A}ut(D_{n})$ that preserves intervals in $D_{n}$ except for the identity $(0,1)$.
\end{proposition}

\begin{proof}
It is not complicated but a bit long if we want to do all the cases. We will do only one case and leave the others to the reader. Let us look, for instance, at the right interval between $(p,1)$ and $(q,1)$
$$
^r\textbf{int}((p,1),(q,1))=(q-p,1).
$$
If we apply $(l,k)$ we get $(l,k)(p,1)=(kp,1)$ and $(l,k)(q,1)=(kq,1)$. Then
$$
^r\textbf{int}((kp,1),(kq,1))=(k(q-p),1)
$$
and if $(l,k)$ is interval preserving we must have $k=1$. 

If $k=1$ we can look at
$$
^r\textbf{int}((p',-1),(q',1))=(p'-q',-1)
$$
and after the left action of $(l,1)$ we get
$$
^r\textbf{int} ((p'+l,-1),(q',-1))=(p'+l-q',-1).
$$
If $(l,1)$ is interval preserving we must have $l=0$. It works the same for left intervals.
\qed
\end{proof}

We did not find any other interval preserving operation, that is why we keep the following definition for non-trivial homometry. 

\begin{definition}
We say that two sets in $D_{n}$ are \textnormal{non-trivially} right- (resp. left-) homometric if there are right- (resp. left-) homometric and not linked by left (resp. right) translation.
\end{definition}

The homometric sets given in the introduction are non-trivially homometric. In what follows, when we do not specify, 'homometric sets' will mean 'non-trivial homometric sets'.

\section{Link with the $T/I$ and the $PLR$-groups in the Case $n=12$}

As mentioned in \cite{Popoff}, the actions of the $T/I$-group and the $PLR$-group on the set of major and minor triads can be considered as the left and right actions of $D_{12}$ on $S$, but also as the actions of $D_{12}$ on itself. To understand why, we use a (non canonical) bijection between $D_{12}$ and $S$. The element $(s,+1)$ of $D_{12}$ will be identified to the major triad whose root is $s \in \mathbb{Z}_{12}$, whereas the element $(s,-1)$ of $D_{12}$ will be considered as the minor triad whose root is $s \in \mathbb{Z}_{12}$. For instance $(0,1)$ corresponds to $C$, $(0,-1)$ corresponds to $c$, $(8,-1)$ corresponds to $a^{\flat}$, and so on. The set $\{ c, D^{\flat}, E^{\flat}, e, a^{\flat} \}$ given in the introduction can be then identified with the set $\{ (0,-1), (1,1), (3,1), (4,-1), (8,-1) \}$.

If we consider the left action of $D_{12}$ on itself we have the following isomorphism between $D_{12}$ (as the acting group) and the $T/I$-group: $(p,+1)$ corresponds to $T_p$ and $(p,-1)$ corresponds to $I_{p-5}$. For instance the image of $C$ by $T_7$ is calculated in $D_{12}$ as the element corresponding to $(7,1)(0,1)=(7,1)$, i.e the major chord $G$. Similarly the image of $C$ by $I_2$ is calculated to be the element corresponding to $(7,-1)(0,1)=(7,-1)$, i.e. the minor chord $g$.

If we consider the right action of $D_{12}$ on itself we have the following bijection between $D_{12}$ and the $PLR$-group: $P$ corresponds to $(0,-1)$, $L$ corresponds to $(4,-1)$ and $R$ corresponds to $(9,-1)$. For instance $P(C)$ is calculated as the element corresponding to $(0,1)(0,-1)=(0,-1)$ which is $c$, $L(d^{\flat})$ corresponds to $(1,-1)(4,-1)=(9,1)$ which is $A$, and $R(F)$ to $(4,1)(9,-1)=(1,-1)$ which is $d^{\flat}$. We obtain on Figure \ref{figright} a new version of Figure \ref{figleftright1} with the left and the right intervals in $D_{12}$.
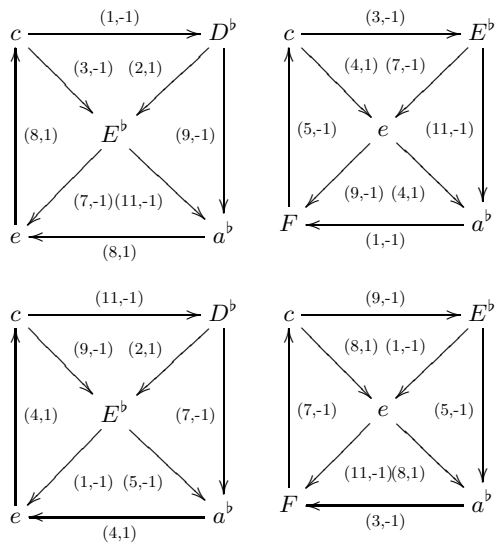
\begin{figure}[t!]
\[
\xymatrix{
c \ar[dr]^{\scalebox{0.7}{(3,-1)}} \ar[rr]^{\scalebox{0.7}{(1,-1)}} & & D^{\flat} \ar[dl]_{\scalebox{0.7}{(2,1)}} \ar[dd]_{\scalebox{0.7}{(9,-1)}} \\
& E^{\flat} \ar[dl]^{\scalebox{0.7}{(7,-1)}} \ar[dr]_{\scalebox{0.7}{(11,-1)}} & \\
e \ar[uu]_{\scalebox{0.7}{(8,1)}} & & a^{\flat} \ar[ll]^{\scalebox{0.7}{(8,1)}} }
\quad
\xymatrix{
c \ar[dr]^{\scalebox{0.7}{(4,1)}} \ar[rr]^{\scalebox{0.7}{(3,-1)}} & & E^{\flat} \ar[dl]_{\scalebox{0.7}{(7,-1)}} \ar[dd]_{\scalebox{0.7}{(11,-1)}} \\
& e \ar[dl]^{\scalebox{0.7}{(9,-1)}} \ar[dr]_{\scalebox{0.7}{(4,1)}} & \\
F \ar[uu]_{\scalebox{0.7}{(5,-1)}} & & a^{\flat} \ar[ll]^{\scalebox{0.7}{(1,-1)}} }
\]
\[
\xymatrix{
c \ar[dr]^{\scalebox{0.7}{(9,-1)}} \ar[rr]^{\scalebox{0.7}{(11,-1)}} & & D^{\flat} \ar[dl]_{\scalebox{0.7}{(2,1)}} \ar[dd]_{\scalebox{0.7}{(7,-1)}} \\
& E^{\flat} \ar[dl]^{\scalebox{0.7}{(1,-1)}} \ar[dr]_{\scalebox{0.7}{(5,-1)}} & \\
e \ar[uu]_{\scalebox{0.7}{(4,1)}} & & a^{\flat} \ar[ll]^{\scalebox{0.7}{(4,1)}} }
\quad
\xymatrix{
c \ar[dr]^{\scalebox{0.7}{(8,1)}} \ar[rr]^{\scalebox{0.7}{(9,-1)}} & & E^{\flat} \ar[dl]_{\scalebox{0.7}{(1,-1)}} \ar[dd]_{\scalebox{0.7}{(5,-1)}} \\
& e \ar[dl]^{\scalebox{0.7}{(11,-1)}} \ar[dr]_{\scalebox{0.7}{(8,1)}} & \\
F \ar[uu]_{\scalebox{0.7}{(7,-1)}} & & a^{\flat} \ar[ll]^{\scalebox{0.7}{(3,-1)}} }
\]
\caption{Left (top) and right (bottom) intervals in $D_{12}$ for the two sets $\{ c, D^{\flat}, E^{\flat}, e, a^{\flat} \}$ and $\{ c, E^{\flat}, e, F, a^{\flat} \}$.}
\label{figright}
\end{figure}

Since we are interested in the general case of homometry in $D_{n}$, we will work from now on using the general point of view of semi-direct products and their group elements $(l, \eta)$.

\section{Homometry in $D_{n}$: Formulas and Enumeration}

In order to avoid confusions we will adopt the notation '$\mathcal{A}$' for subsets in $D_{n}$, and the notation '$A$' for subsets in $\mathbb{Z}_{n}$. Given the form of the group elements of $D_{n}$ as pairs $(l, \eta)$, a set $\mathcal{A}\in D_{n}$ is the disjoint union of two (possibly empty) subsets $\mathcal{A}_{+}$ and $\mathcal{A}_{-}$, with $\mathcal{A}_{+} = \{(l, \eta) \in \mathcal{A} \mid \eta=1\}$, and $\mathcal{A}_{-} = \{(l, \eta) \in \mathcal{A} \mid \eta=-1\}$. For instance in $D_{12}$ the set
\begin{eqnarray*}
\mathcal{A}=\{ c, D^{\flat}, E^{\flat}, e, a^{\flat} \}=\{ (0,-1), (1,1), (3,1), (4,-1), (8,-1) \}
\end{eqnarray*}
given in the introduction is the union of the major chords $\{ D^{\flat}, E^{\flat} \}=\{ (1,1),(3,1)\}$ and  the minor chords $\{ c, e, a^{\flat} \}=\{ (0,-1),(4,-1),(8,-1)\}$. Let $\pi: D_{n} \longrightarrow \mathbb{Z}_{n}$ be the projection on the first factor, i.e. $\pi((l, \eta)) = l$. We define the sets $A_+$ and $A_-$ as $A_+=\pi(\mathcal{A}_{+})=\{\pi((l, \eta)) \mid (l, \eta) \in \mathcal{A}_{+}\}$, and $A_-=\pi(\mathcal{A}_{-})=\{\pi((l, \eta)) \mid (l, \eta) \in \mathcal{A}_{-}\}$. In the example above, we have $A_+= \{ 1,3 \} \subset \mathbb{Z}_{12}$ and $A_-=\{ 0,4,8 \} \subset \mathbb{Z}_{12}$. Remark that we have $\pi(\mathcal{A})=\pi(\mathcal{A}_+) \cup \pi(\mathcal{A}_-)=A_+ \cup A_-$. If there is no ambiguity we will call $A= \pi(\mathcal{A})$.

The purpose of the following theorem is to give a characterization of homometry in $D_{n}$ using \textbf{iv} and \textbf{ifunc}.

\begin{theorem}
\label{thmcond}
Two sets $\mathcal{A}$ and $\mathcal{B}$ in $D_{n}$ are homometric for the right action if and only if the following two equations hold:
\begin{eqnarray}
\label{equrightcond}
& \left\{
    \begin{array}{lc}
        \textbf{iv}(A_+) + \textbf{iv}(A_-) = \textbf{iv}(B_+) + \textbf{iv}(B_-) &\\
        \textbf{ifunc}(A_+,A_-)=\textbf{ifunc}(B_+,B_-). &
    \end{array}
\right. & 
\end{eqnarray}

Two sets $\mathcal{A}$ and $\mathcal{B}$ in $D_{n}$ are homometric for the left action if and only if the following two equations hold:
\begin{eqnarray}
\label{equleftcond}
& \left\{
    \begin{array}{lc}
        \textbf{iv}(A_+) + \textbf{iv}(A_-) = \textbf{iv}(B_+) + \textbf{iv}(B_-) &\\
        \textbf{ifunc}(I_0A_+,A_-)=\textbf{ifunc}(I_0B_+,B_-). &
    \end{array}
\right. & 
\end{eqnarray}
\end{theorem}

\begin{proof}
Let $\mathcal{A},\mathcal{B}$ be two right homometric sets in $D_{n}$. Let us recall that 
\begin{equation}
^r\textbf{int}((k_1, \epsilon_1),(k_2, \epsilon_2))=((k_2-k_1)/\epsilon_1, \epsilon_2/\epsilon_1)
\end{equation}
for $(k_1, \epsilon_1)$ and $(k_2, \epsilon_2)$ in $\mathcal{A}$. We then have to consider two cases, corresponding to the two equations of (\ref{equrightcond}).

In the first case, $\epsilon_2/\epsilon_1=1$, i.e. $\epsilon_1=\epsilon_2$. Then,
\begin{itemize}
\item{either $\epsilon_1=1=\epsilon_2$ in which case we have $^r\textbf{int}((k_1, \epsilon_1),(k_2, \epsilon_2))=(k_2-k_1,1)$
for $k_1$ and $k_2$ in $A_+$, meaning that we have to calculate \textbf{iv}($A_+$) to obtain all the intervals of that type, or}
\item{$\epsilon_1=-1=\epsilon_2$ then $^r\textbf{int}((k_1, \epsilon_1),(k_2, \epsilon_2))=(k_1-k_2,1)$ for $k_1, k_2$ in $A_-$, meaning that we have to calculate \textbf{iv}($A_-$) to obtain all the intervals of that type.}
\end{itemize}
We then must have $\textbf{iv}(A_+) + \textbf{iv}(A_-) = \textbf{iv}(B_+) + \textbf{iv}(B_-)$.

In the second case, $\epsilon_2/\epsilon_1=-1$, i.e. $\epsilon_1=-\epsilon_2$. Then,
\begin{itemize}
\item{either $\epsilon_1=1,\epsilon_2=-1$, thus we have $^r\textbf{int}((k_1, \epsilon_1),(k_2, \epsilon_2)) =(k_2-k_1,-1)$ for $k_1 \in A_+$ and $k_2 \in A_-$, meaning that we have to calculate \textbf{ifunc}($A_+,A_-$) to obtain all the intervals of that type, or}
\item{$\epsilon_1=-1,\epsilon_2=1$, then $^r\textbf{int}((k_1, \epsilon_1),(k_2, \epsilon_2))=(k_1-k_2,1)$ for $k_1 \in A_-$ and $k_2 \in A_+$, meaning that we have to calculate again \textbf{ifunc}($A_+,A_-$) to obtain all the intervals of that type.}
\end{itemize} 
We then must have $\textbf{ifunc}(A_+,A_-)=\textbf{ifunc}(B_+,B_-)$, which leads to the two equations of (\ref{equrightcond}). Reciprocally if two sets verify (\ref{equrightcond}), then they have the same right interval vector.
This works similarly with left intervals, the difference being that when $\epsilon_2/\epsilon_1=-1$, we obtain $^l\textbf{int}((k_1, \epsilon_1),(k_2, \epsilon_2))=(k_1+k_2,-1)$, and we thus calculate \textbf{ifunc}($I_0A_+,A_-$). \hfill \qed
\end{proof}

We can notice that the first equations in both (\ref{equrightcond}) and (\ref{equleftcond}) are identical, but the second one shows an important difference. For left homometry it is symmetric between $A_+$ and $A_-$ (and between $B_+$ and $B_-$), whereas it is not for right homometry, due to the fact that \textbf{ifunc}($I_0A_+,A_-$)=\textbf{ifunc}($I_0A_-,A_+$).

We give now the triviality conditions for homometric sets in $D_{n}$. It has no immediate application but it will be useful later.

\begin{proposition}
\label{proptriv}
Two sets $ \mathcal{A}$ and $ \mathcal{B}$ in $D_{n}$ are:
\begin{list}{-}{}
\item trivially right-homometric if and only if there exists $p\in \mathbb{Z}_{n}$ such that
\begin{equation}
\label{righttriv}
\left( T_pA_+=B_+ \mbox{ and } T_pA_-=B_-\right) \mbox{ or } (I_pA_+=B_- \mbox{ and } I_pA_-=B_+);
\end{equation}
\item trivially left-homometric if and only if there exists $p\in \mathbb{Z}_{n}$ such that
\begin{equation}
\label{lefttriv}
(T_pA_+=B_+ \mbox{ and } T_{-p}A_-=B_-) \mbox{ or } (T_pA_+=B_- \mbox{ and } T_{-p}A_-=B_+).
\end{equation}
\end{list}
\end{proposition}

\begin{proof}
We study each case, beginning with right homometry. 

If $\mathcal{A}$ and $\mathcal{B}$ are trivially right-homometric we have two cases: $(p,1)\mathcal{A}=\mathcal{B}$ or $(p,-1)\mathcal{A}=\mathcal{B}$. 
\begin{list}{-}{}
\item If we have $(p,1)\mathcal{A}=\mathcal{B}$, reminding that $(p,1)(a,1)=(p+a,1)$ and that $(p,1)(a,-1)=(p+a,-1)$ we have
\begin{align*}
(p,1)\mathcal{A}=\mathcal{B} &\Longleftrightarrow (p,1)\mathcal{A}_+=\mathcal{B}_+ \mbox{ and } (p,1)\mathcal{A}_-=\mathcal{B}_-\\
                              &\Longleftrightarrow T_pA_+=B_+ \mbox{ and } T_pA_-=B_-.
\end{align*}
\item If we have $(p,-1)\mathcal{A}=\mathcal{B}$, reminding that $(p,-1)(a,1)=(p-a,-1)$ and that $(p,-1)(a,-1)=(p-a,1)$ we have
\begin{align*}
(p,-1)\mathcal{A}=\mathcal{B} &\Longleftrightarrow (p,-1)\mathcal{A}_+=\mathcal{B}_- \mbox{ and } (p,-1)\mathcal{A}_-=\mathcal{B}_+\\
                               &\Longleftrightarrow I_pA_+=B_- \mbox{ and } I_pA_-=B_+.
\end{align*}
\end{list}

If $\mathcal{A}$ and $\mathcal{B}$ are trivially left-homometric we have two cases: $\mathcal{A}(p,1)=\mathcal{B}$ or $\mathcal{A}(p,-1).=\mathcal{B}$. 
\begin{list}{-}{}
\item If we have $\mathcal{A}(p,1)=\mathcal{B}$, reminding that $(a,1)(p,1)=(a+p,1)$ and that $(a,-1)(p,1)=(a-p,-1)$ we have
\begin{align*}
\mathcal{A}(p,1)=\mathcal{B} &\Longleftrightarrow \mathcal{A}_+(p,1)=\mathcal{B}_+ \mbox{ and } \mathcal{A}_-(p,1)=\mathcal{B}_-\\
                              &\Longleftrightarrow T_pA_+=B_+ \mbox{ and } T_{-p}A_-=B_-.
\end{align*}
\item If we have $\mathcal{A}(p,-1)=\mathcal{B}$, reminding that $(a,1)(p,-1)=(a+p,-1)$ and that $(a,-1)(p,-1)=(a-p,1)$ we have
\begin{align*}
\mathcal{A}(p,-1)=\mathcal{B} &\Longleftrightarrow \mathcal{A}_+(p,-1)=\mathcal{B}_- \mbox{ and } \mathcal{A}_-(p,-1)=\mathcal{B}_+\\
                               &\Longleftrightarrow T_pA_+=B_- \mbox{ and } T_{-p}A_-=B_+.
\end{align*}
\end{list}
It covers all the cases.
\qed
\end{proof}

Let us denote by $I$ the inversion operator in $D_{n}$. As $(k,1)^{-1}=(-k,1)$ and $(k,-1)^{-1}=(k,-1)$ for $k\in \mathbb{Z}_{n}$, it is easy to calculate $I( \mathcal{A})$ for a set $ \mathcal{A}$ in $D_{n}$: we just have to take the inverse of $A_+$ and keep $A_-$ unchanged. For example for $\mathcal{A}=\lbrace (0,-1), (1,1), (3,1), (4,-1), (8,-1) \rbrace \in D_{12}$, we obtain 
$$
I(\mathcal{A})= \lbrace (0,-1), (11,1), (9,1), (4,-1), (8,-1) \rbrace.
$$
Using the inversion we can switch from right-homometric sets to left-homometric sets and reciprocally, which is very practical. It is describred in the following theorem.

\begin{theorem}
\label{thmtau}
Let $\mathcal{A}$ and $\mathcal{B}$ be two sets in $D_{n}$. $\mathcal{A}$ and $\mathcal{B}$ are non-trivially right-homometric if and only if $I(\mathcal{A})$ and $I(\mathcal{B})$ are non-trivially left-homometric.
\end{theorem}

\begin{proof}
Let $\mathcal{A}$ and $\mathcal{B}$ be two right homometric sets. It means $^r\text{iv}(A)=^r\text{iv}(B)$. First we remark that
\begin{align*}
^r\textbf{int}((k_1, \epsilon_1),(k_2, \epsilon_2))&=(k_1,\epsilon_1)^{-1}(k_2,\epsilon_2)\\
^l\textbf{int}((k_1, \epsilon_1),(k_2, \epsilon_2))&=(k_2,\epsilon_2)(k_1,\epsilon_1)^{-1}
\end{align*}
We will then use a more general writing: an element in $D_{n}$ will be denoted by a single letter $a=(k, \epsilon)$. Thus $^r\textbf{int}(a,a')=a^{-1}a'$.
Then if we fix an element $g \in D_{n}$, we have
\begin{align*}
^r\textbf{iv}(\mathcal{A})(g)&=\lbrace a\in A \mid \exists a'\in A, ^r\textbf{int}(a, a') = g^{-1}\rbrace \mbox{ (we use }\textbf{iv}(A)(g)=\textbf{iv}(A)(g^{-1}))\\
                  &= \lbrace a\in A \mid \exists a'\in A, ^r\textbf{int}(a, a')^{-1} = g\rbrace \\
                  &=\lbrace a\in A \mid \exists a'\in A, ^l\textbf{int}(a^{-1}, a'^{-1})= g \rbrace  \\
                  &= ^l\textbf{iv}(I(\mathcal{A}))(g).
\end{align*}
Consequently $^l\textbf{iv}(I(\mathcal{A}))=^l\textbf{iv}(I(\mathcal{B}))$ i.e. $I(\mathcal{A})$ and $I(\mathcal{B})$ are left-homometric. It works similarly if $\mathcal{A}$ and $\mathcal{B}$ are left-homometric.

Besides, if $\mathcal{A}$ and $\mathcal{B}$ are right translated, there exists $g$ in $D_{n}$ such that $Ag=B$. We conclude immediately that $g^{-1}I(\mathcal{A})=I(\mathcal{B})$ i.e. $I(\mathcal{A})$ and $I(\mathcal{B})$ are left translated one from the other. It works similarly if $\mathcal{A}$ and $\mathcal{B}$ are left translated one from the other. Then we conclude the result concerning the triviality.
\qed
\end{proof}

\begin{corollary}
For all $n\in \mathbb{N}$, the number of right homometric sets in $ D_{n}$ is equal to the number of left homometric sets. Besides, we can deduce all the left homometric sets from the right homometric sets (and reciprocally) with the inversion $I$.
\end{corollary}

This result is useful when we deal with the problem of enumeration of homometric sets in $ D_{n}$ (which is an open problem as in $ \mathbb{Z}_{n}$, cf. \cite{Jed}) since we only have to do the calculation for right (or left) homometric sets and not both of them. It also shows a kind of symmetry between left and right homometry, but in fact they work very differently concerning a specific point given in the following proposition, which is simple but important for the issues we consider after.
\begin{proposition}
\label{propproj}
If $\mathcal{A}$ and $\mathcal{B}$ are right-homometric in $D_{n}$, then their first projections $A=\pi_1 (\mathcal{A})$ and $B=\pi_1 (\mathcal{B})$ are homometric in $\mathbb{Z}_{n}$. Besides, if the homometry is trivial in $D_{n}$, the homometry is also trivial between the projections in $\mathbb{Z}_{n}$.
\end{proposition}

\begin{proof}
If $\mathcal{A}$ and $\mathcal{B}$ are right-homometric in $D_{n}$, we have
\begin{align*}
\textbf{iv}(A)&= \textbf{iv}(A_+) + \textbf{iv}(A_-) + \textbf{ifunc}(A_+,A_-) + \textbf{ifunc}(A_-,A_+) \\
              &= \textbf{iv}(B_+) + \textbf{iv}(B_-) + \textbf{ifunc}(B_+,B_-) + \textbf{ifunc}(B_-,B_+)\\
              &=\textbf{iv}(B).
\end{align*} 
Then $A$ and $B$ are homometric in $\mathbb{Z}_{n}$. 

If $\mathcal{A}$ and $\mathcal{B}$ are trivially right-homometric we have several cases as usual (cf. Prop. \ref{proptriv}). If 
$$
T_pA_+=B_+ \mbox{ and } T_{p}A_-=B_-
$$
we obtain
$$
T_p(A_+ \cup A_-)=B_+ \cup B_- \Longrightarrow T_pA=B,
$$
then $A$ and $B$ are trivially homometric in $\mathbb{Z}_{n}$. For the second case $(p,-1)\mathcal{A}=\mathcal{B}$ we obtain $I_pA=B$ then $A$ and $B$ are also trivially homometric in $\mathbb{Z}_{n}$.
\qed
\end{proof}

In other words, homometry for the right action in $ D_{n}$ "implies" homometry in $ \mathbb{Z}_{n}$. However left homometry does not, as the pair of sets $$(\{ (0, 1), (1, -1), (2, 1), (5, -1), (7, -1)\},\{ (0, 1), (1, -1), (6, 1), (7, -1), (8, 1)\})$$ in $D_{10}$ shows. These sets are left homometric but their projection $\{ 0,1, 2,5,7 \}$ and $\{ 0,1,6,7,8\}$ are not homometric in $\mathbb{Z}_{10}$. The Proposition \ref{propproj} raises the question of whether left- or right-homometric sets in $D_{n}$ can be found from homometric sets in $\mathbb{Z}_{n}$. In other words, can we split two homometric sets $A$ and $B$ in $\mathbb{Z}_{n}$ into subsets $(A_+,A_-)$ and $(B_+, B_-)$ such that the corresponding sets in $D_{n}$ are homometric? This question will be considered in the following section with the definition of the concept of \textit{lift}. 

Before moving to this section we give some computational results concerning the enumeration of homometric sets in $D_{n}$. By a brute-force approach, a complete enumeration of such sets was performed, with cardinality equal to 4, 5,  6 or 7 for $n \leq 18$. The result is summarized in Tab. \ref{tabenum}. The first homometric pair appears for $n=8$, $p=4$ (a direct result of Proposition \ref{propproj} and known results about homometry in $\mathbb{Z}_{n}$). Homometric $t$-uples with $t>2$ also exist, the first triple appearing for $n=12$ and $p=5$ (interestingly, the first homometric triple in $\mathbb{Z}_{n}$ only appears for $n=16$ and $p=6$). The first simultaneously right- and left- homometric pair appears for $n=8$ and $p=4$. 

Table \ref{tabmusicalform} gives a complete list of left and right homometric pairs and triples written in musical form for $n=12$ with $p=4$ and $p=5$. Notice that the first two pairs with $p=5$ in this are both left- and right-homometric.

\begin{table}
\caption{Left and right homometric sets in $D_{12}$ in musical form.}
\label{tabmusicalform}
\begin{center}
\begin{tabular}{|c|c|c|c|}
 \hline 
  $N=12$  & Type  & Homometric sets for the action  & Homometric sets  for the action  \tabularnewline
  &  & of the $T/I$-group (left action) & of the $PLR$-group (right action) \\
  \hline 
  $p=4$ & Pairs & $\lbrace C,d,e^{\flat},G^{\flat} \rbrace \&  \lbrace C,c,g^{\flat},A \rbrace$ & $\lbrace C,c,e^{\flat},G^{\flat} \rbrace \&  \lbrace C,c,E^{\flat},g^{\flat} \rbrace$ \\
      &  & $\lbrace C,d^{\flat},e,G^{\flat} \rbrace \&  \lbrace C,d^{\flat},g,A \rbrace$ & $\lbrace C,d^{\flat},e,G^{\flat} \rbrace \&  \lbrace C,d^{\flat},E^{\flat},g \rbrace$ \\
      &  & $\lbrace C,d,f,G^{\flat} \rbrace \&  \lbrace C,d,a^{\flat},A \rbrace$ & $\lbrace C,d,f,G^{\flat} \rbrace \&  \lbrace C,d,E^{\flat},a^{\flat} \rbrace$ \\
  \hline
  $p=5$ & Pairs & $\lbrace C,c,d,E,A^{\flat} \rbrace \&  \lbrace C,d,e,E,A^{\flat} \rbrace$ & $\lbrace C,c,d,E,A^{\flat} \rbrace \&  \lbrace C,d,e,E,A^{\flat} \rbrace$ \\
  &  & $\lbrace C,d^{\flat},e^{\flat},E,A^{\flat} \rbrace \&  \lbrace C,e^{\flat},E,f,A^{\flat} \rbrace$ & $\lbrace C,d^{\flat},e^{\flat},E,A^{\flat} \rbrace \&  \lbrace C,e^{\flat},E,f,A^{\flat} \rbrace$ \\
  & &$\lbrace C,c,d^{\flat},f,G^{\flat} \rbrace \&  \lbrace C,c,g^{\flat},G,B \rbrace$ & $\lbrace C,c,d^{\flat},f,G^{\flat} \rbrace \&  \lbrace C,c,D^{\flat},F,g^{\flat} \rbrace$ \\
  & &$\lbrace C,c,e^{\flat},f,G^{\flat} \rbrace \&  \lbrace C,c,E^{\flat},g^{\flat},B \rbrace$ & $\lbrace C,c,e,f,G^{\flat} \rbrace \&  \lbrace C,c,D^{\flat},g^{\flat},A^{\flat} \rbrace$ \\
  & &$\lbrace C,c,D^{\flat},g^{\flat},A^{\flat} \rbrace \&  \lbrace C,c,g^{\flat},G,A^{\flat} \rbrace$ & $\lbrace C,c,E,F,g^{\flat} \rbrace \&  \lbrace C,c,E,g^{\flat},B \rbrace$ \\
  & &$\lbrace C,d^{\flat},D^{\flat},g,A^{\flat} \rbrace \&  \lbrace C,d^{\flat},g,G,A^{\flat} \rbrace$ &  $\lbrace C,d^{\flat},E,F,g \rbrace \&  \lbrace C,d^{\flat},E,g,B \rbrace$\\
  & &$\lbrace C,D^{\flat},d,a^{\flat},A^{\flat} \rbrace \&  \lbrace C,d,G,a^{\flat},A^{\flat} \rbrace$ & $\lbrace C,d,E,F,a^{\flat} \rbrace \&  \lbrace C,d,E,a^{\flat},B \rbrace$ \\
  & &$\lbrace C,D^{\flat},e^{\flat},A^{\flat},a \rbrace \&  \lbrace C,e^{\flat},G,A^{\flat},a \rbrace$ & $\lbrace C,e^{\flat},E,F,a \rbrace \&  \lbrace C,e^{\flat},E,a,B \rbrace$ \\
  \cline{2-4}
  & Triples & $\lbrace C,c,d,e^{\flat},G^{\flat} \rbrace \&  \lbrace C,c,D,g^{\flat},B^{\flat} \rbrace$ & $\lbrace C,c,d,e,G^{\flat} \rbrace \&  \lbrace C,c,D,E,g^{\flat} \rbrace$ \\
  & & $\&  \lbrace C,c,g^{\flat},A^{\flat},B^{\flat} \rbrace$ &  $\&  \lbrace C,c,D,g^{\flat},B^{\flat} \rbrace$ \\
  & &$\lbrace C,d^{\flat},e^{\flat},f,G^{\flat} \rbrace \&  \lbrace C,d^{\flat},D,g,B^{\flat} \rbrace$ & $\lbrace C,d^{\flat},e^{\flat},f,G^{\flat} \rbrace \&  \lbrace C,d^{\flat},D,E,g \rbrace$ \\
  & & $\&  \lbrace C,d^{\flat},g,A^{\flat},B^{\flat} \rbrace$ &  $\&  \lbrace C,d^{\flat},D,g,B^{\flat} \rbrace $\\
  \hline
\end{tabular}
\end{center}
\end{table}

\begin{table}[t!]
\caption{Table with the number of homometric pairs, triples, $t$-uples in $D_{n}$ for different values of $n$ and $p$.}
\label{tabenum}
\begin{center}
\begin{tabular}{|c|c|c|c|}
  \hline 
  Cardinality & $D_{n}$   & Homometric sets for   & Simulataneous right and   \tabularnewline
  &  & the right/ left action & left homometric sets\\
  \hline 
  $p=4$ &  $n=8$ & 2 pairs & 2 pairs \\
      \cline{2-4}
      & $n=12$ & 3 pairs & 3 pairs \\
      \cline{2-4}
      & $n=16$ & 4 pairs & 4 pairs \\
  \hline
  $p=5$ & $n=8$ & 12 pairs & 12 pairs \\
      \cline{2-4}
      & $n=10$ & 20 pairs & 20 pairs \\
      \cline{2-4}
      & $n=12$ & 8 pairs/2 triples & 8 pairs/2 triples \\
      \cline{2-4}
      & $n=14$ & 21 pairs & 21 pairs \\
      \cline{2-4}
      & $n=15$ & 15 pairs & 15 pairs \\
      \cline{2-4}
      & $n=16$ & 40 pairs & 40 pairs \\
      \cline{2-4}
      & $n=18$ & 30 pairs/3 triples & 30 pairs/3 triples \\     
  \hline 
  $p=6$ & $n=8$ & 30 pairs/3 quadruples & 30 pairs/3 quadruples \\
      \cline{2-4}
      & $n=9$ & 54 pairs/3 triples & 54 pairs/3 triples \\
      \cline{2-4}
      & $n=10$ & 70 pairs & 30 pairs \\
      \cline{2-4}
      & $n=12$ & 358 pairs & 358 pairs \\
      \cline{2-4}
      & $n=14$ & 252 pairs & 84 pairs \\
      \cline{2-4}
      & $n=15$ & 225 pairs & 225 pairs \\
      \cline{2-4}
      & $n=16$ & 500 pairs/6 quadruples & 500 pairs/6 quadruples \\
      \cline{2-4}
      & $n=18$ & 906 pairs/6 triples & 474 pairs/6 triples \\
  \hline 
  $p=7$ & $n=8$ & 36 pairs & 36 pairs \\
      \cline{2-4}
      & $n=9$ & 63 pairs & 63 pairs \\
      \cline{2-4}
      & $n=10$ & 102 pairs/3 quintuples & 82 pairs/3 quintuples \\
      \cline{2-4}
      & $n=11$ & 55 pairs & 55 pairs \\
      \cline{2-4}
      & $n=12$ & 317 pairs/11 triples/10 quadruples/  & 293 pairs/11 triples/10 quadruples/ \\
      & &  2 sextuples/1 octuple &  2 sextuples/1 octuple \\
      \cline{2-4}
      & $n=13$ & 130 pairs & 78 pairs \\
      \cline{2-4}
      & $n=14$ & 539 pairs & 497 pairs \\
      \cline{2-4}
      & $n=15$ & 405 pairs & 405 pairs \\
      \cline{2-4}
      & $n=16$ & 976 pairs & 912 pairs \\
      \cline{2-4}
      & $n=17$ & 136 pairs & 136 pairs \\
      \cline{2-4}
      & $n=18$ & 1785 pairs/30 triples/27 quadruples & 1623 pairs/30 triples/27 quadruples \\
  \hline
\end{tabular}
\end{center}
\end{table}

\section{The Concept of Lift -- Using the Fourier Transform}

We begin this section with a definition motivated by Prop. \ref{propproj}. $\mathcal{P}(E)$ corresponds to the power set of the set $E$.

\begin{definition}
We call a \textnormal{lift} an application $l: \mathcal{P}(\mathbb{Z}_{n})\longrightarrow \mathcal{P}(D_{n})$ such that $\pi \circ l=id$. We will call \textnormal{lift of a set $A\in \mathbb{Z}_{n}$ for the lift $l$}, the set $l(A)$.
\end{definition}

We could say that $l$ is a way to attribute to each number of a set $+1$ or $-1$. In more formal terms the former question is: given two homometric sets $A$ and $B$ in $\mathbb{Z}_{n}$, is there a lift $l$ such that $l(A)$ and $l(B)$ are (left/right-) homometric in $D_{n}$? 

We use the Fourier transform to express these conditions. Actually the Fourier transform is very convenient when we deal with the functions \textbf{ifunc} and \textbf{iv} for subsets in $\mathbb{Z}_{n}$, as explained in the work of Amiot (\cite{Amiot}). Let us recall that for $A$ and $B$ two subsets in $\mathbb{Z}_{n}$ we have for $t \in \mathbb{Z}_{n}$
\begin{equation}
\textbf{ifunc}(A,B)(t)=\bbbone_{-A} \star \bbbone_{B}(t)= \sum_{k \in \mathbb{Z}_{n}} \bbbone_{A}(k)\bbbone_{B}(t+k).
\end{equation}
If we apply the Fourier transform (we write $\mathcal{F}_{A}:=\mathcal{F}(\bbbone_{A}): t \mapsto \sum_{k \in A} e^{-2i \pi k t/n}$) to this convolution product, we obtain the classical result for $t\in \mathbb{Z}_{n}$
\begin{equation}
\label{equifunc}
\mathcal{F}(\textbf{ifunc}(A,B))(t)=\mathcal{F}_{-A}(t) \mathcal{F}_{B}(t).
\end{equation}
As \textbf{iv}($A$)=\textbf{ifunc}($A,A$) and $\mathcal{F}_{-A}(t)=\overline{\mathcal{F}_{A}(t)}$ we deduce from Eq. (\ref{equifunc}) that $\mathcal{F}(\textbf{iv}(A))= \vert \mathcal{F}_{A} \vert^2$ and we get the well-known characterization of homometry in $\mathbb{Z}_{n}$. For $A$ and $B$ two subsets of $\mathbb{Z}_{n}$:
\begin{equation}
\label{equcaractfour}
A \mbox{ is homometric with } B \Longleftrightarrow \vert \mathcal{F}_{A} \vert =\vert \mathcal{F}_{B} \vert.
\end{equation}

The use of the Fourier transform gives a new formulation of Thm. \ref{thmcond}.

\begin{theorem}
\label{thmcondfourier}
Two sets $\mathcal{A}$ and $\mathcal{B}$ in $D_{n}$ are homometric for the right action if and only if the two following equations hold:
\begin{eqnarray}
\label{equrightcondfourier}
& \left\{
    \begin{array}{lc}
        \vert \mathcal{F}_{A_+}\vert^2 + \vert \mathcal{F}_{A_-} \vert^2 =\vert \mathcal{F}_{B_+}\vert^2 + \vert \mathcal{F}_{B_-} \vert^2 &\\
        \overline{\mathcal{F}_{A_+}}\mathcal{F}_{A_-}=\overline{\mathcal{F}_{B_+}}\mathcal{F}_{B_-}. &
    \end{array}
\right. & 
\end{eqnarray}

Two sets $\mathcal{A}$ and $\mathcal{B}$ in $D_{n}$ are homometric for the left action if and only if the two following equations hold:
\begin{eqnarray}
\label{equleftcondfourier}
& \left\{
    \begin{array}{lc}
       \vert \mathcal{F}_{A_+}\vert^2 + \vert \mathcal{F}_{A_-} \vert^2 =\vert \mathcal{F}_{B_+}\vert^2 + \vert \mathcal{F}_{B_-} \vert^2 &\\
        \mathcal{F}_{A_+}\mathcal{F}_{A_-}=\mathcal{F}_{B_+}\mathcal{F}_{B_-}.&
    \end{array}
\right. & 
\end{eqnarray}
\end{theorem}

Recall that we want to decompose two homometric sets in $ \mathbb{Z}_{n}$ both into two subsets in order to lift them in $D_{n}$. The following proposition gives a special characterization of homometry in $\mathbb{Z}_{n}$ for such a decomposition, using the Fourier transform.

\begin{proposition}
\label{propzrelation}
Let $A$ and $B$ be two sets in $\mathbb{Z}_{n}$ such that $A=A_1\cup A_2$ and $B=B_1\cup B_2$ for some subsets $A_1$, $A_2$, $B_1$ and $B_2$ in $ \mathbb{Z}_{n}$. $A$ and $B$ are homometric if and only if
\begin{equation}
\vert \mathcal{F}_{A_1}\vert^2+\vert \mathcal{F}_{A_2}\vert^2 + 2 \mathcal{R}e(\overline{\mathcal{F}_{A_1}}\mathcal{F}_{A_2})=\vert \mathcal{F}_{B_1}\vert^2+\vert \mathcal{F}_{B_2}\vert^2 + 2 \mathcal{R}e(\overline{\mathcal{F}_{B_1}}\mathcal{F}_{B_2}).
\end{equation}
\end{proposition}

\begin{proof}
We use Eq. (\ref{equcaractfour}) and the fact that 
\begin{align*}
\hspace*{1.8cm} \vert \mathcal{F}_{A} \vert^2=\vert \mathcal{F}_{A_1}+\mathcal{F}_{A_2}\vert^2=\vert \mathcal{F}_{A_1}\vert^2+\vert \mathcal{F}_{A_2}\vert^2 + 2 \mathcal{R}e(\overline{\mathcal{F}_{A_1}}\mathcal{F}_{A_2}).  \hspace*{1.45cm} \qed
\end{align*}
\end{proof}

We can now give the main result of this paper which solves the question of lift in a special case.

\begin{theorem}
\label{thmmajeur}
Let $A$ and $B$ be two homometric sets in $\mathbb{Z}_{n}$ such that $A=A_1\cup A_2$ and $B=B_1\cup B_2$ with \textbf{iv}$(A_1)$=\textbf{iv}$(B_1)$ and \textbf{iv}$(A_2)$=\textbf{iv}$(B_2)$. We can always lift $A$ and $B$ into (non trivial) right-homometric sets in $D_{n}$.
\end{theorem}

\begin{proof}
Let $A$ and $B$ be two homometric subsets verifying the conditions of the theorem. We know from Prop. \ref{propzrelation} that
$$
\vert \mathcal{F}_{A_1}\vert^2+\vert \mathcal{F}_{A_2}\vert^2 + 2 \mathcal{R}e(\overline{\mathcal{F}_{A_1}}\mathcal{F}_{A_2})=\vert \mathcal{F}_{B_1}\vert^2+\vert \mathcal{F}_{B_2}\vert^2 + 2 \mathcal{R}e(\overline{\mathcal{F}_{B_1}}\mathcal{F}_{B_2}).
$$
As \textbf{iv}$(A_1)$=\textbf{iv}$(B_1)$ and \textbf{iv}$(A_2)$=\textbf{iv}$(B_2)$ we deduce 
\begin{equation}
\label{equegalite}
\vert \mathcal{F}_{A_1} \vert=\vert \mathcal{F}_{B_1} \vert \mbox{ and } \vert \mathcal{F}_{A_2} \vert=\vert \mathcal{F}_{B_2} \vert,
\end{equation}
so we get
$$
\mathcal{R}e(\overline{\mathcal{F}_{A_1}}\mathcal{F}_{A_2})=\mathcal{R}e(\overline{\mathcal{F}_{B_1}}\mathcal{F}_{B_2}).
$$

We remark also that $\vert \overline{\mathcal{F}_{A_1}}\mathcal{F}_{A_2} \vert= \vert \overline{\mathcal{F}_{B_1}}\mathcal{F}_{B_2} \vert$ thanks to (\ref{equegalite}). We obtain finally the two following equations:
\begin{eqnarray}
\label{condzrelation}
& \left\{
    \begin{array}{cc}
       \mathcal{R}e(\mathcal{F}_{A_1}\overline{\mathcal{F}_{A_2}})=\mathcal{R}e(\mathcal{F}_{B_1}\overline{\mathcal{F}_{B_2}}) &\\
        \vert \mathcal{F}_{A_1}\overline{\mathcal{F}_{A_2}}\vert =\vert \mathcal{F}_{B_1}\overline{\mathcal{F}_{B_2}}\vert. &
    \end{array}
\right. & 
\end{eqnarray}

These equations are of the form $\mathcal{R}e(z)=\mathcal{R}e(z')$ and $\vert z\vert = \vert z'\vert$, which implies $z=z'$ or $z=\overline{z'}$ i.e.
\begin{align*}
&\mathcal{F}_{A_1}\overline{\mathcal{F}_{A_2}}=\mathcal{F}_{B_1}\overline{\mathcal{F}_{B_2}}\\ \mbox{ or } &\mathcal{F}_{A_1}\overline{\mathcal{F}_{A_2}}=\overline{\mathcal{F}_{B_1}}\mathcal{F}_{B_2}.
\end{align*}

In the first case ($\mathcal{F}_{A_1}\overline{\mathcal{F}_{A_2}}=\mathcal{F}_{B_1}\overline{\mathcal{F}_{B_2}}$) if we choose $A_+=A_2$, $A_-=A_1$, $B_+=B_2$ and $B_-=B_1$ we get right homometric sets in the dihedral group since (\ref{equrightcondfourier}) is verified. In the second case ($\mathcal{F}_{A_1}\overline{\mathcal{F}_{A_2}}=\overline{\mathcal{F}_{B_1}}\mathcal{F}_{B_2}$) if we choose $A_+=A_2$, $A_-=A_1$, $B_+=B_1$ and $B_-=B_2$ we get also right homometric sets. Thanks to Prop. \ref{propproj} we know that this homometry is not trivial. \qed
\end{proof}

This result proves not only the existence of right homometric lifts but gives also a way to build these lifts, which is very practical. We will describe it in a concrete example. Before we give two interesting corollaries of Thm. \ref{thmmajeur}.

\begin{corollary}
\label{thm4n}
In $\mathbb{Z}_{4N}$ ($ N\geq 2$), we can always lift homometric sets with cardinality equal to 4 into right homometric sets in $D_{4N}$.
\end{corollary}

\begin{proof}
Rosenblatt (\cite{Rosenblatt}) proved that if $A$ and $B$ are homometric in $\mathbb{Z}_{n}$ with $\sharp(A)=\sharp(B)=4$, they are of the following two types:
\begin{list}{-}{}
\item In $\mathbb{Z}_{4N}: \exists a \in \lbrace 1,2,...,N-1 \rbrace, N\geq 2,$
\begin{equation}
\label{equros}
A= \lbrace0,a,a+N,2N \rbrace \textnormal{ and } B= \lbrace0,a,N,2N+a \rbrace
\end{equation}
\item in $\mathbb{Z}_{13N}$: in that case we do not have any general formulation.
\end{list}
In the case of (\ref{equros}), if we choose $A_1=\lbrace0,2N \rbrace, A_2=\lbrace a,a+N \rbrace, B_1= \lbrace a,2N+a \rbrace$ and $B_2=\lbrace0,N \rbrace$, the conditions of Thm. \ref{thmmajeur} are verified since 
$$
B_1= T_a A_1 \Longrightarrow  \textbf{iv}(A_1)=\textbf{iv}(B_1)
$$
and 
$$
B_2= T_{-a} A_2 \Longrightarrow \textbf{iv}(A_2)=\textbf{iv}(B_2) 
$$
\qed
\end{proof}

\begin{corollary}
\label{liftZ12}
We can lift all the homometric sets in $\mathbb{Z}_{12}$ into right homometric sets in $D_{12}$.
\end{corollary}

\begin{proof}
In \cite{Goyette} Goyette classifies the homometric sets in $\mathbb{Z}_{n}$ in four types. This classification is based on the existence of cyclic subsets contained in the sets we consider. Goyette says that homometric sets in $\mathbb{Z}_{12}$ are only of type 1 and 2, which are two types that satisfy the conditions of Thm. \ref{thmmajeur}. For more details refer to \cite{Goyette}.
\qed
\end{proof}

In order to give an illustration of these two corollaries and an explicit construction of lifts, we will consider a concrete example that we already mentioned in the introduction, based on the two famous "all interval tetrachords" $S_1=\lbrace 0,1,4,6 \rbrace$ and $S_2=\lbrace 0,1,3,7 \rbrace$ in $\mathbb{Z}_{12}$. These sets are homometric of the form of Eq. (\ref{equros}) with $a=1$ (here $N=3$). From the proofs of Cor. \ref{thm4n} and Thm. \ref{thmmajeur} we know that if we choose $A_+=\lbrace 0,6\rbrace$, $A_-=\lbrace 1,4\rbrace$, $B_+=\lbrace 1,7\rbrace$ and $B_-=\lbrace 0,3\rbrace$ we can lift $S_1$ and $S_2$ in $D_{12}$ into the two right homometric sets
\begin{align*}
&\mathcal{S}_1= \lbrace (0,1),(1,-1),(4,-1),(6, 1)\rbrace \\
&\mathcal{S}_2= \lbrace (0,-1),(1,1),(3,-1),(7, 1)\rbrace
\end{align*}

From a musical point of view the homometric melodies $S_1=\lbrace C,D^{\flat},E,G^{\flat} \rbrace$ and $S_2=\lbrace C,D^{\flat},E^{\flat}, G \rbrace$ lift into the right homometric chord sequences
\begin{align*}
&\mathcal{S}_1= \lbrace C,d^{\flat},e,G^{\flat}\rbrace \\
&\mathcal{S}_2= \lbrace c,D^{\flat},e^{\flat},G\rbrace
\end{align*}

Corollary \ref{liftZ12} is particularly interesting when we consider musical applications. Every homometric melodies in $\mathbb{Z}_{12}$ can be transformed into right homometric chord progressions whose roots form the former melodies.

One interesting question is to find simultaneously right and left homometric sets. We already gave an example in the introduction with the sets:
\begin{align*}
&\mathcal{D}=\lbrace (0,-1),(1,1),(3,1),(4,-1),(8,-1)\rbrace=\lbrace c,D^{\flat},E^{\flat},e,a^{\flat} \rbrace \\
&\mathcal{E}=\lbrace (0,-1),(3,1),(4,-1),(5,1),(8,-1)\rbrace=\lbrace c,E^{\flat},e,F,a^{\flat}\rbrace
\end{align*}
in $ D_{12}$. These sets have in fact a specificity, presented in the following proposition.

\begin{proposition}
\label{propdoublehom}
If two sets $ \mathcal{A}$ and $ \mathcal{B}$ in $D_{n}$ are such that $I_0A_+=A_+$ and $I_0B_+=B_+$ (or $I_0A_-=A_-$ and $I_0B_-=B_-$), then the equations for right homometry and left homometry in $D_{n}$ are the same. In other words
$$
\mathcal{A}\mbox{ and } \mathcal{B} \mbox{ are right homometric } \Longleftrightarrow  \mathcal{A} \mbox{ and } \mathcal{B} \mbox{ are left homometric. }
$$
\end{proposition}

\begin{proof}
If $I_0A_+=A_+$ and $I_0B_+=B_+$, we have $\overline{\mathcal{F}_{A_+}}=\mathcal{F}_{A_+}$ and $\overline{\mathcal{F}_{B_+}}=\mathcal{F}_{B_+}$. Thus the conditions of right and left homometry are the same (Eq. \ref{equrightcondfourier} and Eq. \ref{equleftcondfourier} are identical). If $I_0A_-=A_-$ and $I_0B_-=B_-$ it works exactly the same.
\qed
\end{proof}

The above sets $\mathcal{E}$ and $\mathcal{D}$ correspond to lifts of the sets $D=\lbrace 0,1,3,4,8\rbrace$ and $E=\lbrace 0,3,4,5,8\rbrace$ in $\mathbb{Z}_{12}$, with $D_+=\lbrace 1,3\rbrace$, $D_-=\lbrace 0,4,8\rbrace$, $E_+=\lbrace 3,5\rbrace$ and $E_-=\lbrace 0,4,8\rbrace$. Here $D_-=E_-$ and $I_0D_-=D_-$ hence we are in the situation of Prop. \ref{propdoublehom}. Remark also that $E_+=T_2D_+$. 

The first pair of Tab. \ref{tabmusicalform} with $p=5$ verifies also the conditions of Prop. \ref{propdoublehom}:
$$
(\lbrace C,c,d,E,A^{\flat} \rbrace , \lbrace C,d,e,E,A^{\flat} \rbrace).
$$ 

In the next section we present a way to extend the musical interpretation of homometry in the dihedral group.

\section{Musical Interpretation of Homometry in $D_{n}$ for all $n$}

We gave a musical interpretation to homometry in $ D_{n}$ only for the case $n=12$, as homometry between sets of major and minor triads : we identified in $ D_{12}$ the pairs $(k, +1)$ with the major triads and the pairs $(k, -1)$ with the minor triads. We will generalize this interpretation for all $n$.

If we call $\langle X\rangle$ the set generated by the action of the $T/I$-group on a chosen chord $X$ -- i.e. $\langle X\rangle=\lbrace X,T_1X, ..., T_{11}X, I_0X, ...,I_{11}X\rbrace$ -- we have a natural action of the $T/I$-group on $\langle X\rangle$ 
\begin{eqnarray*}
&\lambda :&T/I \longrightarrow Sym(\langle X\rangle)\\
& &g \longmapsto (x \mapsto gx)
\end{eqnarray*}
We will suppose that this action is simply transitive. We know from \cite{Fiore} and \cite{Popoff1} that this action is essentially the same as left multiplication ($g(hX)=(gh)X$), and that we can define a second action 
\begin{eqnarray*}
&\rho :&T/I \longrightarrow Sym(\langle X\rangle)\\
& &g \longmapsto (hX \mapsto hg^{-1}X),
\end{eqnarray*}
which is the same as right multiplication. The group $\rho(T/I)$ is the dual group to $\lambda(T/I)$, the functions $T/I \longrightarrow \langle X\rangle$ and $\rho$ depend on $X$, but the group $\rho(T/I)$ does not. If we choose $X=C$-major, $\rho(T/I)$ is in fact the $PLR$-group which is, as we know, the dual group to the $T/I$-group. The operations $P$, $L$ and $R$ correspond to right multiplication by $I_7$, $I_{11}$ and $I_4$. With this point of view we generalize the action of the $PLR$-group on the set $S$, seeing it as a right action of the $T/I$-group on $\langle X\rangle$. In this perspective, we obtain on Fig. \ref{figcomposingmusic1} a new version of Fig. \ref{figleftright1} with the sets $A= \lbrace c,D^\flat,E^\flat,e,a^\flat\rbrace$ and $B=\lbrace c,E^\flat,e,F,a^\flat\rbrace$.

\begin{figure}
\[
\xymatrix{
I_7C \ar[dr]^{I_{10}} \ar[rr]^{I_8} & & T_1C \ar[dl]_{T_2} \ar[dd]_{I_4} \\
& T_3C \ar[dl]^{I_2} \ar[dr]_{I_6} & \\
I_{11}C \ar[uu]_{T_8} & & I_3C \ar[ll]^{T_8} }
\quad
\xymatrix{
I_7C \ar[dr]^{T_4} \ar[rr]^{I_{10}} & & T_3C \ar[dl]_{I_2} \ar[dd]_{I_6} \\
& I_{11}C \ar[dl]^{I_4} \ar[dr]_{T_4} & \\
T_5C \ar[uu]_{I_0} & & I_3C \ar[ll]^{I_8} }
\]

\[
\xymatrix{
I_7C \ar[dr]^{I_4} \ar[rr]^{I_6} & & T_1C \ar[dl]_{T_2} \ar[dd]_{I_2} \\
& T_3C \ar[dl]^{I_8} \ar[dr]_{I_0} & \\
I_{11}C \ar[uu]_{T_4} & & I_3C \ar[ll]^{T_4} }
\quad
\xymatrix{
I_7C \ar[dr]^{T_8} \ar[rr]^{I_{4}} & & T_3C \ar[dl]_{I_8} \ar[dd]_{I_0} \\
& I_{11}C \ar[dl]^{I_6} \ar[dr]_{T_8} & \\
T_5C \ar[uu]_{I_2} & & I_3C \ar[ll]^{I_{10}} }
\]
\caption{Left-intervals (top) and right-intervals (bottom) in the $T/I$-group for the two sets $A=\lbrace c, D^{\flat}, E^{\flat}, e, a^{\flat} \rbrace$ and $B=\lbrace c, E^{\flat}, e, F, a^{\flat} \rbrace$ with $X=C$-major.}
\label{figcomposingmusic1}
\end{figure}
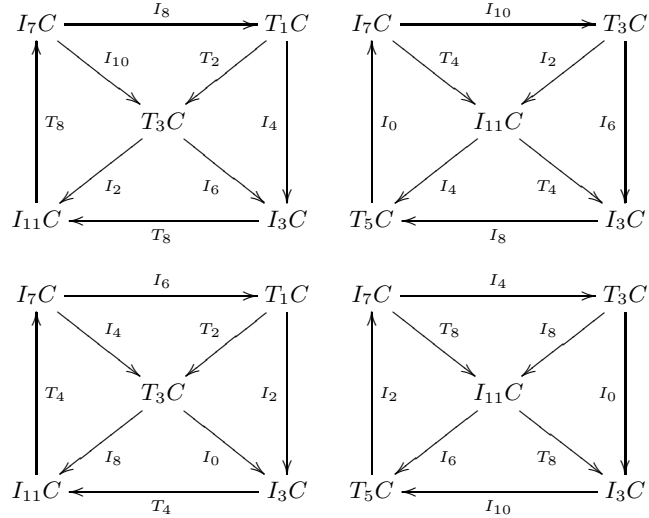
With other choices of $X$ such that the action of the $T/I$-group on $X$ is simply transitive, we deduce from the same example that $\lbrace X, I_8X, T_8X, T_4X, I_{10}X \rbrace$ and $\lbrace X, I_{10}X, T_8X, I_0X, T_4X \rbrace$ are two left-homometric sets (top of Fig. \ref{figcomposingmusic2}). We can then generalize our interpretation of homometry in the dihedral groups to chords with more than 3 notes, and also to unclassified chords. For instance with $X=C^7$, we obtain the two left-homometric chord sequences of Fig. \ref{figcomposingmusic2} (bottom), which contain dominant and half-diminished chords. With $X=[0,1,5]=[C,D^\flat,F]$ we obtain the two left-homometric sets
\begin{align*}
&\lbrace [C,D^\flat,F],[A^\flat,A,E],[A^\flat,A,D^\flat],[E,F,A],[B^\flat,B,G^\flat] \rbrace \\
&\lbrace [C,D^\flat,F],[B^\flat,B,G^\flat],[A^\flat,A,D^\flat],[C,D^\flat,A^\flat],[E,F,A] \rbrace.
\end{align*}
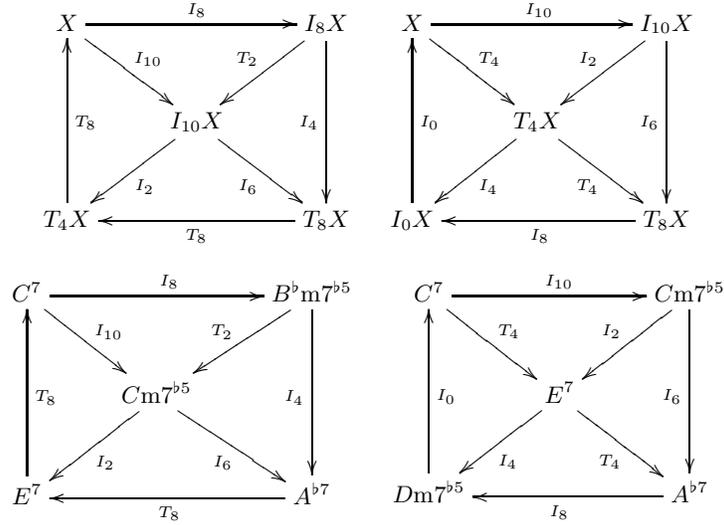
\begin{figure}
\[
\xymatrix{
X \ar[dr]^{I_{10}} \ar[rr]^{I_8} & & I_8X \ar[dl]_{T_2} \ar[dd]_{I_4} \\
& I_{10}X \ar[dl]^{I_2} \ar[dr]_{I_6} & \\
T_4X \ar[uu]_{T_8} & & T_8X \ar[ll]^{T_8} }
\quad
\xymatrix{
X \ar[dr]^{T_4} \ar[rr]^{I_{10}} & & I_{10}X \ar[dl]_{I_2} \ar[dd]_{I_6} \\
& T_4X \ar[dl]^{I_4} \ar[dr]_{T_4} & \\
I_0X \ar[uu]_{I_0} & & T_8X \ar[ll]^{I_8} }
\]

\[
\xymatrix{
C^7 \ar[dr]^{I_{10}} \ar[rr]^{I_8} & & B^{\flat}\mbox{m}7 ^{\flat 5} \ar[dl]_{T_2} \ar[dd]_{I_4} \\
& C\mbox{m}7^{\flat 5} \ar[dl]^{I_2} \ar[dr]_{I_6} & \\
E^7 \ar[uu]_{T_8} & & A^{\flat 7} \ar[ll]^{T_8} }
\quad
\xymatrix{
C^7 \ar[dr]^{T_4} \ar[rr]^{I_{10}} & & C\mbox{m}7^{\flat 5} \ar[dl]_{I_2} \ar[dd]_{I_6} \\
& E^7 \ar[dl]^{I_4} \ar[dr]_{T_4} & \\
D\mbox{m}7^{\flat 5} \ar[uu]_{I_0} & & A^{\flat 7} \ar[ll]^{I_8} }
\]
\caption{(Top) Two left-homometric sets for $X$ such that the $T/I$-group acts simply transitively on $X$. (Bottom) Application with $X=C^7$: we obtain two left-homometric chord progressions containing.}
\label{figcomposingmusic2}
\end{figure}

Besides, we can also obtain a bijection with the dihedral group $ D_{12}$, choosing $(0,1)=X$, $(k,1)=T_kX$, $(0,-1)=I_0X$ and $(k,-1)=I_kX=T_kI_0X$. It allows us to give a musical interpretation to homometry in the dihedral groups for all $n$. Indeed the above identifications are valid for all $n$ and give a bijection between the $T/I$-group in $\mathbb{Z}_{n}$ and the dihedral group $D_{n}$. Thus the left- and the right-homometry in $D_{n}$ can be interpreted as the left- and the right-homometry coming from the left and the right actions of the $T/I$-group on some $\langle X\rangle$ for all $n$. Musically it allows us to use microtonality.

\section{Conclusion}

The present paper concerns homometry in the non-commutative dihedral groups. It incorporates results coming from both \cite{Genuys1} and \cite{Genuys2}. The left and the right actions of the dihedral group on itself can be seen as the action of the $T/I$-group and the $PLR$-group on the set $S$ of major and minor triads. Consequently we interpreted homometry in $D_{12}$ as homometry between sets of triads in $S$. As mentioned in the last section, this interpretation can however be extended to sets on which the $T/I$-group acts simply transitively, which allows us to widen the musical applications of this non-commutative homometry.

We presented important properties (including the notion of \textit{lift}) linking homometry in $\mathbb{Z}_{n}$ and homometry in $D_{n}$: right-homometry in $D_{n}$ implies homometry in $\mathbb{Z}_{n}$, and conversely in some cases we can always lift homometric sets from $\mathbb{Z}_{n}$ to $D_{n}$. We also mentioned cases where we can build homometric lifts both for the left and the right actions. However some problems are still unsolved:
\begin{list}{-}{}
\item{we did a complete enumeration of homometric sets in $D_{n}$ only for small values of $n$ and $p$ (cardinality);}
\item{we could find other properties concerning left homometry. In particular it would be interesting to find general cases where we can build left-homometric lifts.}
\end{list}
Finally we want to mention the possibility to use a similar approach when studying homometry in other semi-direct products. For instance we present in \cite{Genuys2} some equivalent results in the time-spans group $(\mathbb{R},+) \rtimes (\mathbb{R}_+^*,.)$ (in this group sets can be interpreted as musical rhythms) introduced by Lewin.

\bibliographystyle{splncs}

\section*{Acknowledgment}
\small{The author thanks Alexandre Popoff for his fruitful and efficient collaboration to this work.}

\end{document}